\documentclass[11pt,reqno]{amsart}
\usepackage{amsmath,amsxtra,latexsym,amsthm,amssymb,amscd,pb-diagram}
\usepackage{mathrsfs,mathabx,amsfonts}
\usepackage[colorlinks=true,linkcolor=red,citecolor=blue]{hyperref}
\usepackage[margin=1in]{geometry}

\usepackage{bm}
\usepackage{color}
\usepackage{makecell,multirow,diagbox}
\usepackage{mathtools}
\usepackage[displaymath, mathlines]{lineno}

\usepackage[displaymath, mathlines]{lineno}
\usepackage{tikz}
\usepackage{graphicx}
\usepackage{epsfig}
\usepackage{array}
\usepackage{enumitem}
\usepackage{float}
\usepackage{graphicx}
\usepackage{epsfig}
\usepackage{array}
\usepackage{enumitem}

\newtheorem{definition}{Definition}[section]
\newtheorem{theorem}{Theorem}[section]
\newtheorem{proposition}{Proposition}[section]
\newtheorem{lemma}{Lemma}[section]
\newtheorem{corollary}{Corollary}[section]
\newtheorem{remark}{Remark}[section]

\newcommand{\R}{\mathbb{R}}

\newcommand{\ity}{\infty}

\begin{document}
\title[Weighted semilinear damped wave equations]{Critical exponent for semilinear damped wave equations with weighted
nonlinear terms and data from Sobolev spaces of negative
order}

\subjclass{35A01, 35B33, 35B44, 35L15, 35L71}
\keywords{Weighted semilinear damped wave equation, Critical
exponent, Sobolev spaces of negative order, Lifespan estimates}
\thanks{$^* $\textit{Corresponding author:} Tuan Anh Dao (anh.daotuan@hust.edu.vn)}

\maketitle

\centerline{\scshape \textbf{Dinh Van Duong}}
{\footnotesize
	\centerline{Faculty of Mathematics and Informatics, Hanoi University of Science and Technology}
	\centerline{No.1 Dai Co Viet road, Hanoi, Vietnam}
	\centerline{Email: vanmath2002@gmail.com}}
\medskip

\centerline{\scshape \textbf{Tuan Anh Dao}}
{\footnotesize
	\centerline{Faculty of Mathematics and Informatics, Hanoi University of Science and Technology}
	\centerline{No.1 Dai Co Viet road, Hanoi, Vietnam}
	\centerline{Email: anh.daotuan@hust.edu.vn}}

% \linenumbers

\begin{abstract}
   In this paper, we would like to study the critical exponent for semilinear damped wave equations with the nonlinearity terms of Coulomb-type singularities $|x|^{-\alpha} |u(t,x)|^p$  and the initial data belonging to Sobolev spaces of negative order $\dot{H}^{-\beta}$. Precisely, we obtain a critical exponent
$$p_{\rm c}(\alpha,\beta,n): = 1 + \frac{4-2\alpha}{n+2\beta} $$
for $1 \leq n \leq 4$ and $ 0 \leq \alpha, \beta < n/2,$ by proving the global (in time) existence of small data solutions when $p \geq p_{\rm c}(\alpha,\beta,n)$ and the blow-up result for weak solutions in finite time even for small data if $1 < p < p_{\rm c}(\alpha,\beta,n)$. Furthermore, we are going to provide lifespan estimates for solutions when a blow-up phenomenon occurs.
\end{abstract}

\tableofcontents

%=================================================================================={Introduction}	
\section{Introduction}
Let us consider the following Cauchy problem for semilinear damped wave equations:
\begin{equation} \label{Main.Eq.1}
\begin{cases}
u_{tt}(t,x) -\Delta u (t,x) + u_t(t,x)= \mathcal{N}(u(t,x)), &\quad x\in \R^n,\, t > 0, \\
u(0,x)= \varepsilon u_0(x), \,\, u_t(0,x) = \varepsilon u_1(x), &\quad x \in \mathbb{R}^n, \\
\end{cases}
\end{equation}
where the nonlinear terms $\mathcal{N}(u(t,x)) = |x|^{-\alpha}|u(t,x)|^p$ are referred to as weighted power-type nonlinearities, exhibiting Coulomb-type singularities at the origin.
Additionally, we consider the initial data $$(u_0, u_1) \in \dot{H}^{-\beta}(\mathbb{R}^n) \times \dot{H}^{-\beta}(\mathbb{R}^n)$$ with $\beta \geq 0$ and the positive constant $\varepsilon$
describes its size. For $r \in (1, \infty)$ and $\beta \in \mathbb{R}$, the Sobolev spaces $\dot{H}_r^{-\beta}$ are defined by (see, for instance, \cite[section 1.3.3]{Garafakos2014}) 
\begin{align*}
    \dot{H}_r^{-\beta} := \bigg\{\varphi: \varphi \in \mathcal{Z}' \text{ satisfies } \|\varphi\|_{\dot{H}_r^{-\beta}} := \big\||\nabla|^{-\beta} \varphi\big\|_{L^r} = \big\|\frak{F}^{-1}(|\xi|^{-\beta} \widehat{\varphi}(\xi))(t,\cdot)\big\|_{L^r} < +\infty\bigg\}.
\end{align*}
Here $\mathcal{Z}'$ stands for the topological dual space to the subspace of the Schwartz space $\mathcal{S}$ consisting of functions with $\partial_{\xi}^{\gamma} \hat{\varphi}(0) = 0$ for all multi-indices $\gamma$, namely, $\mathcal{Z}'$  is the
factor space $\mathcal{S}'/\mathcal{P}$, where $\mathcal{P}$ is the space of all polynomials. Moreover, the spaces $\dot{H}_r^{-\beta}$ are called the Sobolev spaces of negative order for $\beta > 0$. We recall the other special types of nonlinearities that are singular at the origin. The case $n=3$ and $\mathcal{N}(u(t,x))= -|x|^{-1} u(t,x) $  is known as the Coulomb potential. The case $n\geq 1$ and $\mathcal{N}(u(t,x)) = |x|^{-2} u(t,x)$ is known as the square-inverse potential (see \cite{Baras1984} for heat equations, \cite{Brezis2005} for elliptic equations, and \cite{Burq2003} for Schr\"odinger equations). The nonlinear terms of the form $\mathcal{N}(u(t,x)) = |x|^{-\alpha} u(t,x)|u(t,x)|^p$ are
also considered for elliptic equations and $p$-Laplace equations (see \cite{Garcia1998} and the
references therein). 

Next, the linear problem corresponding to (\ref{Main.Eq.1}) we have in mind is
\begin{equation} \label{Main.Eq.2}
\begin{cases}
u_{tt}(t,x)-\Delta u(t,x)+ u_t(t,x)= 0, &\quad x\in \mathbb{R}^n,\, t > 0, \\
u(0,x) = \varepsilon u_0(x), \, u_t(0,x) = \varepsilon u_1(x) &\quad x\in \mathbb{R}^n.
\end{cases}
\end{equation}
The author in \cite{Matsumura1976} was the first to establish some basic decay estimates for solutions to (\ref{Main.Eq.2}). Subsequently, it has been established that the damped wave equation has a diffusive structure as $t \to \infty$. Since then, many papers have studied sharp $L^m-L^q$ estimates, with $1 \leq m \leq q \leq \infty$, for solutions to (\ref{Main.Eq.2}), for example, \cite{Nishihara2003, DabbiccoEbert2017, Ikeda2019} and the references cited therein.  From this, they conclude that the diffusion phenomenon bridges the
decay properties of solution to the classical damped wave equations (\ref{Main.Eq.2})
and solution to the heat equations.
\medskip

Next, we consider the following semilinear damped wave equations with power nonlinearities:
\begin{equation} \label{Main.Eq.3}
\begin{cases}
u_{tt}(t,x) -\Delta u (t,x) + u_t(t,x)= |u(t,x)|^p, &\quad x\in \R^n,\, t > 0, \\
u(0,x)= \varepsilon u_0(x), \,\, u_t(0,x) = \varepsilon u_1(x), &\quad x \in \mathbb{R}^n, \\
\end{cases}
\end{equation}
where $p > 1$. The papers \cite{TodorovaYordanov2001, IkehataMiyaokaNakatake2004, IkehataTanizawa2005} are obtained a global (in time) weak solution to (\ref{Main.Eq.3}) when the exponent $p$ satisfies
  $
          1+ 2/n < p < +\infty \,\,\text{ if } n = 1,2 \text{ and }
          1 +2/n < p \leq n/(n-2) \,\,\text{ if } n \geq 3,
 $
  by using weighted
  energy methods and assuming the initial data $(u_0,u_1) \in (H^1 \cap L^1)\times (L^2 \cap L^1).$ On the other hand, the nonexistence of general global (in time) small data solutions is proved in \cite{TodorovaYordanov2001} for $1 < p < 1 +2/n$ and in \cite{Zhang2001} for $p = 1+2/n$. Therefore, the quantity
  $$p_{\rm Fuj}(n) := 1+\frac{2}{n} $$ 
 is called the critical exponent for (\ref{Main.Eq.3}), well-known as the Fujita exponent (see more \cite{Palmieri2020, Palmieri2021, GeogievPalmieri2020} for the semilinear damped wave equation in non-Euclidean frameworks). 
Here, the critical exponent is understood as a threshold
between the global (in time) existence of small data weak solutions and the blow-up of solutions even for small data. The diffusion phenomenon connecting the heat equation and the classical damped wave equation (see \cite{Narazaki2004, Nishihara2003}) sheds light on the parabolic nature of classical damped wave models with respect to decay estimates for solutions. Additionally, to derive the critical regularity of nonlinear terms for the semilinear damped wave equations, the paper \cite{Ebert2020}, followed by \cite{Girardi2025,TangDuong2026}, considered the problem (\ref{Main.Eq.3}) with
nonlinearities $\mathcal{N}(u(t,x)) = |u(t,x)|^{p_{\rm Fuj}(n)}\mu(|u(t,x)|)$, where $\mu$ stands for a suitable
modulus of continuity. Moreover, according to the works \cite{LiZhou1995, IkedaOgawa2016, LaiZhou2019}, the sharp lifespan estimates for blow-up solutions to (\ref{Main.Eq.3}) when $1< p \leq p_{\rm Fuj}(n)$ in all spatial dimensions have been investigated. Consequently, they provided the following sharp lifespan estimates for blow-up solutions:
\begin{align*}
    T(\varepsilon) \sim
    \varepsilon^{-\frac{2(p-1)}{2-n(p-1)}} \text{ if } 1 < p < p_{\rm Fuj}(n) \,\,\text{ and } \,\,\log(T(\varepsilon)) \sim 
        \varepsilon^{-\frac{2}{n}} &\text{ if } p = p_{\rm Fuj}(n). 
\end{align*}

Under additional regularity $L^m$ for the initial data, with $m \in [1,2]$, the papers \cite{IkehataOhta2002, DabbiccoEbert2017} and references therein identified the critical exponent for (\ref{Main.Eq.3}) as
$$ p_{\rm Fuj}\left(\frac{n}{m}\right) =  1 + \frac{2m}{n}.$$ However, the authors did not provide conclusions regarding the solutions' properties when $p = p_{\rm Fuj}(n/m)$. The paper \cite{Ikeda2019} proved
the existence of global mild solutions in this critical case for all $m \in (1,2]$ in low dimensional spaces along with some additional conditions.

Continuing an extension of the initial data space $  \dot{H}^{-\beta} \times \dot{H}^{-\beta}$, with $\beta \geq 0$, the authors in \cite{ChenReissig2023} determined that the critical exponent for (\ref{Main.Eq.3}) is 
        \begin{align*}
            p_{\rm c}(0,\beta,n) := p_{\rm Fuj}\left(\frac{n}{2}+\beta\right) = 1+ \frac{4}{n+2\beta}, 
        \end{align*}
by proving the global existence with small data when $p > p_{\rm c}(0,\beta,n)$ and the blow-up solution when $1< p < p_{\rm c}(0,\beta,n)$. In non-Euclidean frameworks, the papers \cite{Dasgupta2024} also found the critical exponent
        $p_{\rm c}(0,\beta, \mathcal{Q}) =  1+ 4/(\mathcal{Q}+2\beta),$ 
with the homogeneous dimensions $\mathcal{Q}=2n+2$ of the Heisenberg group $\mathbb{H}_n$. The paper \cite{Dabbicco2024} demonstrated that the values $p=p_c(0, \beta, n)$ and $p=p_c(0, \beta, \mathcal{Q})$, corresponding to Euclidean space and Heisenberg group cases respectively, belong to the global existence range.  Finally, the reader can refer to a recent preprint \cite{TangDuongPhan2026} concerning the problem (\ref{Main.Eq.3}) with the initial data belonging to pseudo-measure spaces. \medskip

Our main goal is to study (\ref{Main.Eq.1}) for the nonlinearity terms $|x|^{-\alpha} |u(t,x)|^p$  and the initial data belonging to $\dot{H}^{-\beta}$. Specifically, we determine the critical exponent for (\ref{Main.Eq.1}) as follows:

        \begin{align}
            p_{\rm c}(\alpha,\beta,n) := p_{\rm Fuj}\left(\frac{n+2\beta}{2-\alpha}\right) = 1 + \frac{4-2\alpha}{n+2\beta} \label{CriticalExponent}
        \end{align}
for all $1\leq n \leq 4$ and $0 \leq \alpha, \beta < n/2$, that is, the weight $|x|^{-\alpha}$ shifts the critical exponent to the left of $p_c(0, \beta, n)$. In the case $\beta = 0$, the recent paper \cite{NakamuraWadade2019} established the existence of global solutions to (\ref{Main.Eq.1}) for $p \geq p_c(\alpha, 0,n)$, $n=1,2$ and local solutions for $p>1$, $n \geq 1$. Therefore, our results extended those obtained in \cite{ChenReissig2023, Dabbicco2024, NakamuraWadade2019}. 
When there is no singularity $(\alpha = 0)$,  the power type nonlinear terms $|u(t,x)|^p$ could be handled by the standard Sobolev
embeddings $H^1(\mathbb{R}^n)  \hookrightarrow L^r(\mathbb{R}^n)$ for $\max\{0, 1/2-1/n\} \leq 1/r \leq 1/2$ with $(n, r) \ne (2, \infty)$.  When there is the singularity $(\alpha > 0)$, we will use the Caffarelli-Kohn-Nirenberg
inequality (see Proposition \ref{CKN_inquality}) to control its decay rate. Finally, the estimates for lifespan of blow-up solutions to (\ref{Main.Eq.1}) are also addressed in this work. Namely, we denote by $T_{\alpha, \beta}(\varepsilon)$ the so-called lifespan of a local (in time) solution, i.e. its maximal existence time. Then, $T_{\alpha, \beta} \sim +\infty$ when $p \geq p_c(\alpha, \beta, n)$ and 
        $$\varepsilon^{-\frac{4(p-1)}{4-2\alpha-(n+2\beta)(p-1)}} \lesssim T_{\alpha, \beta}(\varepsilon) \lesssim \varepsilon^{-\frac{4(p-1)}{4-2\alpha-(n+2\beta)(p-1)}} (\log(\varepsilon^{-1}))^{\frac{4(p-1)}{4-2\alpha-(n+2\beta)(p-1)}},$$
when $1 < p < p_c(\alpha, \beta,n)$.
Therefore, one can observe that the weight $|x|^{-\alpha}$ contributes to increasing the lifespan of blow-up solutions.

\vspace{0.2cm}
\textbf{Notations:} Throughout this paper, we are going to use the followings:

\begin{itemize}[leftmargin=*]
\item We write $f\lesssim g$ when there exists a constant $C>0$ such that $f\le Cg$, and $f \sim g$ when $g\lesssim f\lesssim g$.

\item We denote $\widehat{w}(t,\xi):= \mathfrak{F}_{x\rightarrow \xi}\big(w(t,x)\big)$ as the Fourier transform with respect to the spatial variable of a function $w(t,x)$. Moreover, $\mathfrak{F}^{-1}$ represents the inverse Fourier transform.

\item We put $\langle x\rangle:= \sqrt{1+|x|^2}$, the so-called Japanese bracket of $x \in \R^n$.

\item As usual, $H^{a}_r$ and $\dot{H}^{a}_r$, with $r \in (1, \infty), a \in \mathbb{R}$, denote potential spaces based on $L^r$ spaces. Here $\langle\nabla\rangle^{a}$ and $|\nabla|^{a}$ stand for the differential operators with symbols $\big<\xi\big>^{a}$ and $|\xi|^{a}$, respectively. In addition, for $r=2$, we write $H^a$ and $\dot{H}^a$ instead of $H_2^a$ and $\dot{H}_2^a$, respectively.
\end{itemize}
\medskip

\textbf{Main results:} Let us state the global (in time) existence of small data solutions, along with the blow-up result, which will be proved in this paper.
    \begin{theorem}[\textbf{Global existence}]\label{Theorem1}
        Let $1 \leq n \leq 4$ and $0 \leq \alpha, \beta < n/2$. We assume that the exponent $p$ satisfies  
        \begin{align}
            p \geq p_c(\alpha, \beta, n). \label{critical_exponent}
        \end{align}
In addition, we also suppose the following condition when $n=3,4$:
        \begin{align}
              1 + \displaystyle\frac{2(\beta-\alpha)}{n} < p \leq \displaystyle\frac{n-2\alpha}{n-2} \quad\text{ for }\quad \,\, 2\alpha + (n-2)\beta < n.
           \label{CKN_condition}
         \end{align}
        Moreover, the initial data satisfies
            $$(u_0, u_1) \in \mathcal{D}(\beta) := (H^1 \cap \dot{H}^{-\beta})\times (L^2 \cap \dot{H}^{-\beta}).$$
       Then, there exists a constant $\varepsilon_0 > 0$ such that for any $\varepsilon \in (0, \varepsilon_0]$, the Cauchy problem \eqref{Main.Eq.1} admits a unique global (in time) solution 
        \begin{align*}
            u \in \mathcal{C}([0, \infty), H^1).
        \end{align*}
      Furthermore, the following estimates hold for all $t > 0$:
        \begin{align*}
            \|u(t,\cdot)\|_{L^2} &\lesssim \varepsilon (1+t)^{-\frac{\beta}{2}} \|(u_0,u_1)\|_{\mathcal{D}(\beta)},\\
            \|\nabla u(t,\cdot)\|_{L^2} &\lesssim \varepsilon (1+t)^{-\frac{1+\beta}{2}} \|(u_0,u_1)\|_{\mathcal{D}(\beta)}.
        \end{align*}
    \end{theorem}

\begin{remark}\label{Remark1.1}
\fontshape{n}
\selectfont
    We now discuss the assumptions appearing in Theorem \ref{Theorem1}. More specifically, the conditions (\ref{CKN_condition}) and $0 \leq \alpha < n/2$ arise from the application of Proposition \ref{CKN_inquality}. Furthermore,  we note that 
    \begin{align}
    p_c(\alpha, \beta,n) > p^*(\alpha, \beta, n) := 1+\frac{2(\beta-\alpha)}{n} \quad\text{ if and only if }\quad  2n-n\beta +2\alpha\beta -  2\beta^2 > 0. \label{Relation1}
    \end{align}
    Therefore, the assumption $p > p^*(\alpha, \beta,n)$ is automatically satisfied when $n=1,2$. For $n=3,4$, we also have the following relation:
    \begin{align}
        p_c(\alpha, \beta,n) \leq \frac{n-2\alpha}{n-2} \quad \text{ if and only if }\quad 4-n+2\beta-2\alpha -2\alpha\beta \geq 0. \label{Relation2}
    \end{align}
    Finally, the condition $0 \leq \beta < n/2$ appears from the use of Proposition \ref{HLS_inequality}.
\end{remark}

 \begin{remark}
    \fontshape{n}
\selectfont
        Let us give some examples for the admissible range of exponents $p$ for the global (in time) existence result.
        \begin{itemize}[leftmargin=*]
            \item When $n = 1,2$ and $0 \leq \alpha, \beta < n/2$, one can see that the exponent $p$ satisfies
            \begin{align*}
               p_{\rm c}(\alpha,\beta,n) \leq p < +\infty.
            \end{align*}
         \item When $n = 3$, using the relations (\ref{Relation1}) and (\ref{Relation2}), we have the exponent $p$ satisfies
            \begin{align}
               p_{\rm c}(\alpha,\beta,3) \leq p \leq 3-2\alpha 
                \quad\text{ if }\quad 
               \begin{cases}
               0 \leq \alpha, \beta < 3/2 ,\\
               1+2\beta-2\alpha-2\alpha\beta \geq 0,\\
               6-3\beta+2\alpha\beta-2\beta^2 > 0
           \end{cases}\label{System1}
            \end{align}
        and 
        \begin{align}
            1+\frac{2(\beta-\alpha)}{3} < p \leq 3-2\alpha  \quad\text{ if }\quad 
               \begin{cases}
               0 \leq \alpha, \beta < 3/2 ,\\
               2\alpha+\beta < 3,\\
               6-3\beta+2\alpha\beta-2\beta^2 \leq 0.
           \end{cases} \label{System2}
        \end{align}
        We can observe that $\alpha \in [0, 1/2), \, \beta = 0$ is a pair satisfying system (\ref{System1}) and $\alpha \in [0, 7/20], \beta = 5/4 $ is a pair satisfying system (\ref{System2}).
        \item When $n=4$, applying again the relation (\ref{Relation1}) and (\ref{Relation2}), we have the exponent $p$ satisfies
        \begin{align}
               p_{\rm c}(\alpha,\beta,4) \leq p \leq 2-\alpha 
                \quad\text{ if }\quad 
               \begin{cases}
               0 \leq \alpha, \beta < 2 ,\\
               \beta-\alpha\beta-\alpha \geq 0,\\
               4-2\beta+\alpha\beta-\beta^2 > 0
           \end{cases}\label{System3}
            \end{align}
        and 
        \begin{align}
            1+\frac{\beta-\alpha}{2} < p \leq 2-\alpha  \quad\text{ if }\quad 
               \begin{cases}
               0 \leq \alpha, \beta < 2 ,\\
               \alpha+\beta < 2,\\
               4-2\beta+\alpha\beta-\beta^2 \leq 0.
        \end{cases}\label{System4}
        \end{align}
        Finally, one can see that $\alpha = 1/2, \, \beta \in [1, (-3+\sqrt{73})/4)$ is a pair satisfying system (\ref{System3}) and $\alpha \in [0, 1/2), \beta = 3/2 $ is a pair satisfying system (\ref{System4}).
        \end{itemize}
    \end{remark}

    \begin{theorem}[\textbf{Blow-up}]\label{Theorem2}
        Let $n \geq 1, \,\, 0 \leq \alpha < 2$ and \,$0 \leq  \beta < n/2$. Consider that $(u_0, u_1) \in \dot{H}^{-\beta} \times \dot{H}^{-\beta} $ satisfies
        \begin{align}\label{condition2.1}
         u_0(x) + u_1(x) \gtrsim \langle x \rangle^{-n(\frac{1}{2}+\frac{\beta}{n})} \log(e+|x|)^{-1}.  
        \end{align}
        Moreover, the exponent $p$ satisfies the following condition:
        \begin{align}\label{condition2.2}
            1 < p < p_{\rm c}(\alpha,\beta,n).
        \end{align}
        Then, there is no global weak solution to \eqref{Main.Eq.1} in $\mathcal{C}([0, \infty), L^2)$.
    \end{theorem}

    \begin{remark}
      \fontshape{n}
\selectfont
       Thanks to the statements of Theorems \ref{Theorem1} and \ref{Theorem2}, under additional regularity $\dot{H}^{-\beta}$ for the initial data and weight $|x|^{-\alpha}$ for the nonlinear term, the critical exponent for (\ref{Main.Eq.1}) is defined by (\ref{CriticalExponent}).  Therefore, both of these factors shift the critical exponent to the left when compared to the exponents $p_c(0, \beta, n)$ or $p_c(\alpha, 0, n)$. For the reader's convenience, the ranges of global (in time) existence and blow-up are illustrated in Figure \ref{fig.zone1} and Figure \ref{fig.zone2}.

%.......................................................................

\begin{figure}[htbp]
    \centering
    % ==================== HÌNH BÊN TRÁI ====================
    \begin{minipage}[b]{0.49\textwidth}
        \centering
        \begin{tikzpicture}[>=latex,xscale=1.1,scale=0.48]
        %=====================================
        \draw[->] (0,0) -- (7.5,0)node[below]{$\beta$};
        \draw[->] (0,0) -- (0,7.5)node[left]{$p$};
        \node[below left] at(0,0){$0$};
        \node[below] at (2.5,1.0) {};

        %=====================================p
        \draw[fill] (3.5,0) circle[radius=1pt];
        \node[below] at (3.5,0) {{\scriptsize $\displaystyle\frac{n}{2}$}};
        \draw[dashed] (3.5,0)--(3.5,7);

        %=====================================q
        \draw[fill] (0,1.5) circle[radius=1pt];
        \node[left] at (0,1.5){{\scriptsize $1$}};
        \draw[dashed] (0,1.5)--(3.5,1.5);

        \draw[fill] (0,4) circle[radius=1pt];
        \node[left] at (0,4){{\scriptsize $1+\displaystyle\frac{4-2\alpha}{n}$}};

        \fill[color=black!10!white] (0,1.5)--(0,4)--(1,3.33333)--(2,2.909090)--(3,2.61538)--(3.5,2.5)--(3.5,1.5)--cycle;
        \fill[color=black!40!white] (0,7)--(0,4)--(1,3.33333)--(2,2.909090)--(3,2.61538)--(3.5,2.5)--(3.5,7)--cycle;

        \draw[domain = 0:3.5, blue,line width=1.0pt]plot(\x,{1+10.5/(3.5+\x)});

        \draw[thin] (0,0)--(0,7.5);
        \draw[thin] (3.5,0)--(3.5,7);
        \draw[thin] (0,1.5)--(3.5,1.5);

        \draw[dashed] (0,2.5)--(3.5,2.5);
        \draw[fill] (0,2.5) circle[radius=1pt];
        \node[left] at (0,2.5){{\scriptsize $1+\displaystyle\frac{2-\alpha}{n}$}};

        \draw[fill] (3.5,2.5) circle[radius=1.5pt];

        %\fill[color=black!10!white] (9.25,8.75)--(8.75,8.75)--(8.75,8.25)--(9.25,8.25)--cycle;
        %\node[right] at (9.4,8.5) {{\footnotesize \text{Blow-up}}};

        \fill[color=black!40!white] (0.25,8.75)--(0.75,8.75)--(0.75,8.25)--(0.25,8.25)--cycle;
        \node[right] at (0.8,8.5) {{\footnotesize \text{Global existence}}};

        \draw[thin, color=blue,line width=1.0pt] (5,4.75)--(5.5,4.75);
        \node[right] at (5.75,4.75) {$p= 1+\frac{4-2\alpha}{n+2\beta}$};
        \end{tikzpicture}
        \caption{The critical exponent in the $\beta-p$ plane when $n =1,2$ and $\alpha$ is fixed.}
        \label{fig.zone1}
    \end{minipage}
    \hfill % Khoảng trống co giãn giữa 2 cột (TUYỆT ĐỐI KHÔNG ĐỂ DÒNG TRỐNG TRONG CODE Ở ĐÂY)
    % ==================== HÌNH BÊN PHẢI ====================
    \begin{minipage}[b]{0.49\textwidth}
        \centering
        \begin{tikzpicture}[>=latex,xscale=1.1,scale=0.48]
        %=====================================
        \draw[->] (0,0) -- (7.5,0)node[below]{$\alpha$};
        \draw[->] (0,0) -- (0,7.5)node[left]{$p$};
        \node[below left] at(0,0){$0$};
        \node[below] at (2.5,1.0) {};

        %=====================================p
        \draw[fill] (3.5,0) circle[radius=1pt];
        \node[below] at (3.5,0) {{\scriptsize $\displaystyle\frac{n}{2}$}};
        \draw[dashed] (3.5,0)--(3.5,7);

        %=====================================q
        \draw[fill] (0,1.5) circle[radius=1pt];
        \node[left] at (0,1.5){{\scriptsize $1$}};
        \draw[dashed] (0,1.5)--(3.5,1.5);

        \draw[fill] (0,5) circle[radius=1pt];
        \node[left] at (0,5){{\scriptsize $1+\displaystyle\frac{4}{n+2\beta}$}};

        \fill[color=black!10!white] (0,1.5)--(0,5)--(3.5,3)--(3.5,1.5)--cycle;
        \fill[color=black!40!white] (0,7)--(0,5)--(3.5,3)--(3.5,7)--cycle;

        \draw[domain = 0:3.5, red,line width=1.0pt]plot(\x,{-4/7*\x+5});

        \draw[thin] (0,0)--(0,7.5);
        \draw[thin] (3.5,0)--(3.5,7);
        \draw[thin] (0,1.5)--(3.5,1.5);

        \draw[dashed] (0,3)--(3.5,3);
        \draw[fill] (0,3) circle[radius=1pt];
        \node[left] at (0,3){{\scriptsize $\displaystyle\frac{n+8}{2n+4\beta}$}};

        \draw[fill] (3.5,3) circle[radius=1.5pt];

        %\fill[color=black!10!white] (9.25,3.75)--(8.75,3.75)--(8.75,3.25)--(9.25,3.25)--cycle;
        %\node[right] at (9.4,3.5) {{\footnotesize \text{Blow-up}}};

         \fill[color=black!10!white] (0.25,8.75)--(0.75,8.75)--(0.75,8.25)--(0.25,8.25)--cycle;
        \node[right] at (0.8,8.5) {{\footnotesize \text{Blow-up}}};%\fill[color=black!40!white] (9.25,2.75)--(8.75,2.75)--(8.75,2.25)--(9.25,2.25)--cycle;
        %\node[right] at (9.4,2.5) {{\footnotesize \text{Global existence}}};

        \draw[thin, color=red,line width=1.0pt] (5,4.75)--(5.5,4.75);
        \node[right] at (5.75,4.75) {$p= 1+\frac{4-2\alpha}{n+2\beta}$};
        \end{tikzpicture}
        \caption{The critical exponent in the $\alpha-p$ plane when $n =1,2$ and $\beta$ is fixed.}
        \label{fig.zone2}
    \end{minipage}
\end{figure}

\end{remark}

\textbf{The remaining part of this paper is organized as follows:} In Section \ref{section3}, we provide the proof of global (in time) existence results for solutions to (\ref{Main.Eq.1}). We establish in Section \ref{Proof of blow-up results} the blow-up result and derive lifespan estimates for
solutions in the subcritical case as well. Finally, in Section \ref{Section4}, we derive several corollaries from the main theorems.

\section{ Global existence}\label{section3}
\subsection{Estimates for linear kernels}
To get started, we can write solutions to the linear problem $(\ref{Main.Eq.2})$ by the formula
\begin{equation*}
 u^{\rm lin}(t,x) = \varepsilon (\mathcal{K}(t,x) +\partial_t \mathcal{K}(t,x))\ast_x u_0(x) + \varepsilon \mathcal{K}(t,x) \ast_x u_1(x),
\end{equation*}
where the Fourier transform of the kernel $\mathcal{K}(t,x)$ is defined by
\begin{align*}
    \widehat{\mathcal{K}}(t,\xi) :=  \begin{cases}
        \vspace{0.3cm}\displaystyle\frac{e^{-\frac{t}{2}}\sinh{\left(t \sqrt{\frac{1}{4} -|\xi|^2}\right)}}{\sqrt{\frac{1}{4}- |\xi|^2}}  &\text{ if } |\xi| \leq \displaystyle\frac{1}{2},\\
        \displaystyle\frac{e^{-\frac{t}{2}}\sin{\left(t \sqrt{|\xi|^2-\frac{1}{4}}\right)}}{\sqrt{|\xi|^2-\frac{1}{4}}} &\text{ if } |\xi| > \displaystyle\frac{1}{2}.
    \end{cases}
\end{align*}
 Now, we are going to prove the important result in the following lemma.

\begin{lemma}\label{LinearEstimates}
    Let $n \geq 1, m \in (1,2], s \geq 0$ and $s+\beta \geq 0$.  Then, the following estimate holds for $j = 0,1$ and $t > 0$:
    \begin{align*}
        \| \partial_t^j \mathcal{K}(t,x)\ast_x \varphi(x)\|_{\dot{H}^s} 
        &\lesssim (1+t)^{-\frac{n}{2}(\frac{1}{m}-\frac{1}{2})-\frac{s+\beta}{2}-j} \|\varphi\|_{\dot{H}_m^{-\beta} \cap H^{s+j-1}}.
        \end{align*}
\end{lemma}
\begin{proof}
From the definition of the kernel $\mathcal{K}(t,x)$, we can get some pointwise estimates in the Fourier space, namely,
$$
|\partial_t^j \widehat{\mathcal{K}}(t,\xi)|\lesssim \begin{cases}
\mathrm{e}^{-ct}+|\xi|^{2j}\mathrm{e}^{-c|\xi|^2t}&\text{ for } |\xi|< \varepsilon^*,\\
\mathrm{e}^{-ct}&\text{ for } \varepsilon^* \leq|\xi|\leq N,\\
|\xi|^{j-1}\mathrm{e}^{-ct}&\text{ for } |\xi|> N,
\end{cases}
$$
with suitable constants $c>0,$  $\varepsilon^* \ll 1$ and $ N \gg 1$. Let $\chi_k= \chi_k(r)$ with $k\in\{\rm L,H\}$ be smooth cut-off functions having the following properties:
\begin{align*}
&\chi_{\rm L}(r)=
\begin{cases}
1 &\quad \text{ if }r\le \varepsilon^*/2 \\
0 &\quad \text{ if }r\ge \varepsilon^*
\end{cases}
\text{ and } \qquad
\chi_{\rm H}(r)= 1 -\chi_{\rm L}(r).
\end{align*}
Using Parseval's formula, we obtain
\begin{align*}
\|\partial_t^j \mathcal{K}(t,x)\ast_x \varphi(x)\|_{\dot{H}^s} =& \||\xi|^s \partial_t^j \widehat{\mathcal{K}}(t,\xi) \widehat{\varphi}(\xi)\|_{L^2} \\
\leq & \|\chi_{\rm L}(|\xi|)|\xi|^s \partial_t^j \widehat{\mathcal{K}}(t,\xi) \widehat{\varphi}(\xi)\|_{L^2} + \| \chi_{\rm H}(|\xi|)|\xi|^s \partial_t^j \widehat{\mathcal{K}}(t,\xi) \widehat{\varphi}(\xi)\|_{L^2}.  
\end{align*}
Concerning the small frequencies part,
applying H\"{o}lder's inequality with $1/m_0 + 1/m' = 1/2$ we gain
\begin{align*}
    \big\|\chi_{\rm L}(|\xi|)|\xi|^{s}\partial_t^j\widehat{\mathcal{K}}(t,\xi) \widehat{\varphi}(\xi)\big\|^2_{L^2} 
     &\leq \big\|\chi_{\rm L}(|\xi|)|\xi|^{s+\beta}\partial_t^j\widehat{\mathcal{K}}(t,\xi)\big\|_{L^{m_0}}^2 \||\xi|^{-\beta}\widehat{\varphi}(\xi)\|_{L^{m'}}^2.
\end{align*}
Thanks to the Hausdorff-Young's inequality, it holds $\||\xi|^{-\beta}\widehat{\varphi}(\xi)\|_{L^{m'}} \leq \|\varphi\|_{\dot{H}_m^{-\beta}}$ with $1/m' + 1/m = 1$  and $m \in (1,2]$. As a result, we obtain
\begin{align*}
     \big\|\chi_{\rm L}(|\xi|)|\xi|^{s}\partial_t^j\widehat{\mathcal{K}}(t,\xi) \widehat{\varphi}(\xi)\big\|_{L^2} &\lesssim \left(\big\|\chi_{\rm L}(|\xi|)|\xi|^{s+\beta+2j} e^{-c|\xi|^{2}(t+1)}\big\|_{L^{m_0}} + e^{-ct} \|\chi_{\rm L}(|\xi|) |\xi|^{s+\beta}\|_{L^{m_0}}\right) \|\varphi\|_{\dot{H}_m^{-\beta}}\\
     &\lesssim (1+t)^{-\frac{n}{2}(\frac{1}{m}-\frac{1}{2})-\frac{s+\beta}{2}-j} \|\varphi\|_{\dot{H}_m^{-\beta}}.
\end{align*}
Moreover, for the large frequencies part, we have immediately the following estimate:
$$
\|\chi_{\rm H}(|\xi|)|\xi|^{s}\partial_t^j \widehat{\mathcal{K}}(t,\xi)\widehat{\varphi}(\xi)\|_{L^2} \lesssim e^{-ct}\|\chi_{\rm H}(|\xi|)|\xi|^{s+j-1} \widehat{\varphi}(\xi)\|_{L^2} \lesssim e^{-ct} \|\varphi\|_{H^{s+j-1}}.
$$ 
The proof of Lemma \ref{LinearEstimates} has been completed.    
\end{proof}

\subsection{Philosophy of our approach}
Thanks to Duhamel's principle, we have the representation formula of solutions to $(\ref{Main.Eq.1})$ as follows:
\begin{align*}
    u(t,x)
    &= u^{\rm lin}(t,x) + u^{\rm non}(t,x),
\end{align*}
where
\begin{align*}
    u^{\rm non}(t,x) := \int_0^t \mathcal{K}(t-\tau,x) \ast_x \mathcal{N}(u(\tau,x)) d\tau.
\end{align*}
Under the assumptions of Theorem \ref{Theorem1}, for $T>0$ we define the following function space:
\begin{align*}
    X(T) :=  \mathcal{C}([0, T), H^1),
\end{align*}
with the norm
\begin{align*}
    \|\varphi\|_{X(T)} &:= \sup _{t \in [0,T]} \bigg\{(1+t)^{\frac{\beta}{2}}\|\varphi(t,\cdot)\|_{L^2} + (1+t)^{\frac{1+\beta}{2}} \| \nabla\varphi(t,\cdot)\|_{L^2} \bigg\}
\end{align*}
and the closed ball
     \begin{align*}
         X(T, K) := \{ \varphi \in X(T), \|\varphi\| \leq K\}
     \end{align*}
     for $K > 0$.
We define the following operator $\Phi$ on the space $X(T)$:
\begin{align*}
   \Phi[u](t,x) := u^{\rm lin}(t,x) + u^{\rm non}(t,x).
\end{align*}
The crux of our proof relies on a couple of the following inequalities for all $u, v \in X(T)$:
    \begin{align}
        \|\Phi[u]\|_{X(T)} &\leq C_1 \varepsilon \|(u_0, u_1)\|_{\mathcal{D}(\beta)} + C_1\|u\|_{X(T)}^p,  \label{Es.Pro2.1}\\
        \|\Phi[u]-\Phi[v]\|_{X(T)} &\leq C_2 \|u-v\|_{X(T)}\left(\|u\|_{X(T)}^{p-1}+\|v\|_{X(T)}^{p-1}\right) \label{Es.Pro2.2}.
    \end{align}
 Then, if we denote that $M$ and $\varepsilon_0 > 0$ satisfy
     \begin{align*}
         M:=  2C_1\|(u_0, u_1)\|_{\mathcal{D}(\beta)} \quad\text{ and }\quad \max\{C_1, 2C_2\} M^{p-1} \varepsilon_0^{p-1} < \frac{1}{2},
     \end{align*}
      $\Phi$ is a contraction mapping on $X(T, M\varepsilon)$ for $\varepsilon \in (0, \varepsilon_0]$. Therefore, by the Banach fixed point theorem, we obtain a unique solution $u^* = \Phi[u^*] \in X(T, M\varepsilon)$ for all $T > 0$. Finally, since $T$ is arbitrary, we conclude that $u^* \in X(\infty, M\varepsilon)$.

\vspace{0.2cm}
To end this subsection, we recall two useful inequalities that will be applied later.
\begin{proposition}[Hardy-Littlewood-Sobolev inequality, see \cite{Lieb1983}]\label{HLS_inequality}
    Let $0 \leq  \beta < n$ and $1 < \theta_1 \leq \theta_2 < \infty$ such that $1/\theta_2 = 1/\theta_1-\beta/n$. Then, there exists a constant $C$ depending only on $\theta_1$ such that 
    \begin{align*}
        \|\varphi\|_{\dot{H}^{-\beta}_{\theta_2}} \leq C \|\varphi\|_{L^{\theta_1}}.
    \end{align*}
        
    \end{proposition}

\begin{proposition}\label{CKN_inquality}
    Let $n \geq 1,\, 1 \leq r < \infty, \,\, 0 \leq \gamma < n/r$.
    Then, the following inequality holds for any nonconstant function $\varphi$:
    \begin{align*}
        \||x|^{-\gamma} \varphi\|_{L^r} \lesssim \| \varphi\|_{L^2}^{1-\sigma(\gamma,r,n)} \| \nabla \varphi \|_{L^2}^{\sigma(\gamma,r,n)} 
    \end{align*}
    where
    \begin{align*}
        \sigma(\gamma,r,n) := \gamma -n\left(\frac{1}{r}-\frac{1}{2}\right) \in (0,1].
    \end{align*}
    \end{proposition}
Proposition \ref{CKN_inquality} is a direct consequence of the Caffarelli-Kohn-Nirenberg inequality (see \cite{Caffarelli1984}).  
%.........................................................
\subsection{Proof of Theorem \ref{Theorem1}} 
Following the approach introduced in the previous section, we first prove estimate (\ref{Es.Pro2.1}). Using Lemma \ref{LinearEstimates} for $j=0,1$, we immediately claim that $u^{\rm lin} \in X(T)$ and
\begin{align*}
    \|u^{\rm lin}\|_{X(T)} \lesssim \varepsilon \|(u_0, u_1)\|_{\mathcal{D}(\beta)}.
\end{align*}
Therefore, in order to conclude the estimate (\ref{Es.Pro2.1}), we need to have
\begin{align}
    \|u^{\rm non}\|_{X(T)} \lesssim \|u\|_{X(T)}^p. \label{Main.Es2}
\end{align}
   The relation $p > 1+ \max\{0,\, 2(\beta-\alpha)/n\}$ leads to the fact that there exists a parameter $\theta$ fulfilling
   \begin{align*}
       \max\bigg\{\frac{2n}{np + 2(\alpha-\beta)}, 1 \bigg\} <  \theta < 2.
   \end{align*}
   We fix $\eta$ so that
   \begin{align}
      \frac{1}{\eta} = \frac{1}{\theta}+\frac{\beta}{n},  \,\,\text{ that is, }\,\, \max\bigg\{\frac{2n}{pn+2\alpha}, 1\bigg\} < \eta <  \frac{2n}{n+2\beta}.\label{RE1}
   \end{align}
    For $k = 0,1$, applying Lemma \ref{LinearEstimates} with $m = \theta, (s,j)=(k,0)$ for $\tau \in [0, t/2)$ and $m=2$, $(s,j)=(k,0)$ for $\tau \in [t/2, t]$ to derive
    \begin{align*}
        \|\nabla^k u^{\rm non}(t,\cdot)\|_{L^2} &\lesssim \int_0^{t/2} (1+t-\tau)^{-\frac{n}{2}(\frac{1}{\theta}-\frac{1}{2})-\frac{k+\beta}{2}} \|\mathcal{N}(u(\tau,\cdot))\|_{\dot{H}_{\theta}^{-\beta}\cap L^2} d\tau\\
        &\hspace{1cm}+\int_{t/2}^t (1+t-\tau)^{-\frac{k}{2}} \|\mathcal{N}(u(\tau,\cdot))\|_{L^2} d\tau.
    \end{align*}
    Thanks to Proposition \ref{HLS_inequality} with $0 \leq \beta < n/2$ and Proposition \ref{CKN_inquality} with $\gamma = \alpha/p,\,\, 0 \leq \alpha < n/2$, we arrive at
    \begin{align}
        \|\mathcal{N}(u(\tau,\cdot))\|_{\dot{H}_{\theta}^{-\beta}} &\lesssim \|\mathcal{N}(u(\tau,\cdot))\|_{L^{\eta}} = \||\cdot|^{-\frac{\alpha}{p}} |u(\tau,\cdot)|\|_{L^{\eta p}}^p \notag\\
        &\lesssim \|u(\tau,\cdot)\|_{L^2}^{p(1-\sigma_1)} \|\nabla u(\tau,\cdot)\|_{L^2}^{p\sigma_1} \lesssim (1+\tau)^{-\frac{n}{2}(\frac{p}{2}-\frac{1}{\eta})-\frac{p\beta}{2}-\frac{\alpha}{2}} \|u\|_{X(T)}^p\notag\\
        &\lesssim (1+\tau)^{-\frac{n}{2}(\frac{p}{2}-\frac{1}{\theta})-\frac{(p-1)\beta}{2}-\frac{\alpha}{2}} \|u\|_{X(T)}^p \notag
    \end{align}
    and
    \begin{align*}
        \|\mathcal{N}(u(\tau,\cdot))\|_{L^2} &= \| |\cdot|^{-\frac{\alpha}{p}} |u(\tau,\cdot)|\|_{L^{2p}}^p \\
        &\lesssim \|u(\tau,\cdot)\|_{L^2}^{p(1-\sigma_2)} \|\nabla u(\tau,\cdot)\|_{L^2}^{p\sigma_2}\\
        &\lesssim (1+\tau)^{-\frac{n}{4}(p-1)-\frac{p\beta}{2}-\frac{\alpha}{2}} \|u\|_{X(T)}^p
    \end{align*}
    for all $t \in [0,T]$,
   provided that the following conditions are satisfied: 
    \begin{align*}
        \sigma_1 := \frac{\alpha}{p}-n\left(\frac{1}{\eta p}-\frac{1}{2}\right) \in (0, 1], \quad \sigma_2:= \frac{\alpha}{p}- n\left(\frac{1}{2p}-\frac{1}{2}\right) \in (0,1].
    \end{align*}
   It is the fact that the assumption (\ref{CKN_condition}) and the relation (\ref{RE1}) lead to them. Therefore, we arrive at

    \begin{align*}
        \|\nabla^k u^{\rm non}(t,\cdot)\|_{L^2} &\lesssim (1+t)^{-\frac{n}{2}(\frac{1}{\theta}-\frac{1}{2})-\frac{\kappa+\beta}{2}} \|u\|_{X(T)}^p\int_0^{t/2} (1+\tau)^{-\frac{n}{2}(\frac{p}{2}-\frac{1}{\theta})-\frac{(p-1)\beta}{2}-\frac{\alpha}{2}} d\tau\\
        &\qquad + (1+t)^{-\frac{n}{4}(p-1)-\frac{p\beta}{2}-\frac{\alpha}{2}} \|u\|_{X(T)}^p \int_{t/2}^t (1+t-\tau)^{-\frac{k}{2}} d\tau\\
        &=: (\mathcal{J}_1(t) + \mathcal{J}_2(t)) \|u\|_{X(T)}^p.
    \end{align*}
    For the term $\mathcal{J}_2(t)$, one can see that
    \begin{align*}
        \mathcal{J}_2(t) \lesssim (1+t)^{-\frac{k+\beta}{2}+1-\frac{n}{4}(p-1)-\frac{(p-1)\beta}{2}-\frac{\alpha}{2}}  \lesssim (1+t)^{-\frac{k+\beta}{2}},
    \end{align*}
    where we note that the condition (\ref{critical_exponent}) implies
    \begin{align*}
        1-\frac{n}{4}(p-1)-\frac{(p-1)\beta}{2}-\frac{\alpha}{2} \leq 0.
    \end{align*}
   Let us proceed with $\mathcal{J}_1(t)$ as follows:
   \begin{itemize}
   [leftmargin=*]
       \item If $$-\frac{n}{2}\left(\frac{p}{2}-\frac{1}{\theta}\right)-\frac{(p-1)\beta}{2}-\frac{\alpha}{2} \leq -1,$$ we immediately have
       \begin{align*}
           \mathcal{J}_1(t) \lesssim (1+t)^{-\frac{n}{2}(\frac{1}{\theta}-\frac{1}{2})-\frac{k+\beta}{2}} \log(e+t) \lesssim (1+t)^{-\frac{k+\beta}{2}},
       \end{align*}
       due to $\theta \in (1,2)$.
   
   \item Otherwise, one has
   \begin{align*}
       \mathcal{J}_1(t) \lesssim (1+t)^{-\frac{k+\beta}{2}+1-\frac{n}{4}(p-1)-\frac{(p-1)\beta}{2}-\frac{\alpha}{2}} \lesssim (1+t)^{-\frac{k+\beta}{2}}.
   \end{align*}
   \end{itemize}
   In summary, we obtain the following estimate:
   \begin{align*}
       \|\nabla^k u(t,\cdot)\|_{L^2} \lesssim (1+t)^{-\frac{k+\beta}{2}} \|u\|_{X(T)}^p
   \end{align*}
   for all $t \in [0,T]$. Therefore, we can conclude the estimate (\ref{Main.Es2}).

\vspace{0.3cm}
\textbf{ Next, we are going to prove the estimate (\ref{Es.Pro2.2})}. For any $u, v \in X(T)$ and similarly to the proof of the estimate (\ref{Es.Pro2.1}), we have
    \begin{align*}
       \|\nabla^k (\Phi[u](t,\cdot) - \Phi[v](t,\cdot))\|_{L^2} & \lesssim \int_0^{t/2} (1+t-\tau)^{-\frac{n}{2}(\frac{1}{\theta}-\frac{1}{2})-\frac{k+\beta}{2}} \|\mathcal{N}(u(\tau,\cdot))-\mathcal{N}(v(\tau,\cdot))\|_{\dot{H}_{\theta}^{-\beta}\cap L^2} d\tau\\
        &\quad+\int_{t/2}^t (1+t-\tau)^{-\frac{k}{2}} \|\mathcal{N}(u(\tau,\cdot))-\mathcal{N}(v(\tau,\cdot))\|_{L^2} d\tau.
    \end{align*}
     By the fact that
     \begin{align*}
         \left|\mathcal{N}(u(\tau,\cdot)) - \mathcal{N}(v(\tau,\cdot))\right| \lesssim |\cdot|^{-\frac{\alpha}{p}}|u(\tau,\cdot)-v(\tau,\cdot)|\left(|\cdot|^{-\frac{\alpha}{p}}|u(\tau,\cdot)| + |\cdot|^{-\frac{\alpha}{p}}|v(\tau,\cdot)|\right)^{p-1}
     \end{align*}
and H\"older's inequality we arrive at
\begin{align*}
    &\|\mathcal{N}(u(\tau,\cdot))-\mathcal{N}(v(\tau,\cdot))\|_{\dot{H}_{\theta}^{-\beta}} \lesssim \|\mathcal{N}(u(\tau,\cdot))-\mathcal{N}(v(\tau,\cdot))\|_{L^\eta}\\
    &\quad\lesssim \big\||\cdot|^{-\frac{\alpha}{p}}|u(\tau,\cdot) - v(\tau,\cdot)|\big\|_{L^{p\eta}} \left(\big\||\cdot|^{-\frac{\alpha}{p}}|u(\tau,\cdot)|\big\|_{L^{p\eta}}^{p-1} +  \big\||\cdot|^{-\frac{\alpha}{p}}|v(\tau,\cdot)|\big\|_{L^{p\eta}}^{p-1}\right)\\
    &\quad \lesssim (1+\tau)^{-\frac{n}{2}(\frac{p}{2}-\frac{1}{\theta})-\frac{(p-1)\beta}{2}-\frac{\alpha}{2}} \|u-v\|_{X(T)}\left(\|u\|_{X(T)}^{p-1} + \|v\|_{X(T)}^{p-1}\right)
\end{align*}
and 
\begin{align*}
    &\|\mathcal{N}(u(\tau,\cdot))-\mathcal{N}(v(\tau,\cdot))\|_{L^2} \lesssim (1+\tau)^{-\frac{n}{4}(p-1)-\frac{p\beta}{2}-\frac{\alpha}{2}} \|u-v\|_{X(T)}\left(\|u\|_{X(T)}^{p-1} + \|v\|_{X(T)}^{p-1}\right)
\end{align*}
for all $t \in [0,T]$.
From these, we can conclude the estimate (\ref{Es.Pro2.2}). Our proof is finished.

%.........................................................

\section{Blow-up result and estimates for lifespan} \label{Proof of blow-up results}

\subsection{Proof of Theorem \ref{Theorem2}}
To get started, let us introduce the following definition of weak solutions to (\ref{Main.Eq.1}).
\begin{definition} \label{defweaksolution_1}
Let $p>1, \alpha \geq 0$,  $T>0$. We say that $u \in \mathcal{C}([0, T), L^2) $ is a local weak solution to \eqref{Main.Eq.1} if for any function $\phi(t,x)= \eta(t) \varphi(x)$, where $\eta \in \mathcal{C}_0^{\infty}([0, T))$,  and for any $\varphi \in \mathcal{C}^{\infty}(\mathbb{R}^n)$ with all derivatives in $L^1 \cap L^\infty$, it holds
\begin{align*}
\int_0^T \int_{\R^n} |x|^{-\alpha} |u(t,x)|^p \phi(t,x)dxdt &= \int_0^T \int_{\R^n} (\partial_t^2 - \Delta + \partial_t)u(t,x) \phi(t,x)dxdt.
\end{align*}
If $T= \ity$, we say that $u \in \mathcal{C}([0, \infty), L^2)$ is a global weak solution to \eqref{Main.Eq.1}. 
\end{definition}

Next, let us introduce the test function $\eta= \eta(t)$ satisfying the following properties:
\begin{align*}
&1.\quad \eta \in \mathcal{C}_0^\ity([0,\ity)) \text{ and }
\eta(t)=\begin{cases}
1 &\text{ for }0 \le t \le 1/2, \\
\text{decreasing} &\text{ for }1/2\le t\le 1, \\
0 &\text{ for }t \ge 1,
\end{cases} & \nonumber \\
&2.\quad \eta^{-\frac{p'}{p}}\big(|\eta'|^{p'}+|\eta''|^{p'}\big) \text{ is bounded, }
\end{align*}
where $p'$ stands for the conjugate number of $p$. In addition, we set $\varphi(x) := \langle x \rangle^{-n-1-\alpha/(p-1)}$ for all $n \geq 1$. Let $R \geq 1$, we define the following test function:
$$ \Phi_{R}(t,x):= \eta_{R}(t) \varphi_R(x), $$
where $\eta_{R}(t):= \eta(R^{-2}t)$ and $\varphi_R(x):= \varphi(R^{-1}x)$. Moreover, we define the functional

$$ \mathcal{I}_{R}:= \int_0^{\ity}\int_{\R^n} |x|^{-\alpha}|u(t,x)|^p \Phi_{R}(t,x) dxdt= \int_{0}^{R^{2}}\int_{\mathbb{R}^n} |x|^{-\alpha}|u(t,x)|^p \Phi_{R}(t,x) dxdt. $$
 We assume that $u \in \mathcal{C}([0, \infty), L^2)$ is a global weak solution to (\ref{Main.Eq.1}) in the sense of Definition \ref{defweaksolution_1}. Then, we replace the function $\phi(t,x)$ by $\Phi_{R}(t,x)$ in Definition \ref{defweaksolution_1} and carry out integration by parts to derive
    \begin{align}
    &\int_0^{\infty}\int_{\mathbb{R}^n} |x|^{-\frac{\alpha}{p}} |u(t,x)|^p \Phi_{R}(t, x) dxdt +\varepsilon\int_{\mathbb{R}^n} (u_0(x)+u_1(x))\varphi_R(x)dx\notag \notag\\
    &\hspace{1cm}= \int_0^{\infty} \int_{\mathbb{R}^n} u(t,x) (\partial_t^2 -\partial_t-\Delta) \Phi_{R}(t,x) dx dt.\label{equation1.3.2}
\end{align}
Applying H\"older's inequality with $1/p + 1/p' = 1$  and using the change of variables $\Tilde{t} = R^{-2} t$, $\Tilde{x} = R^{-1} x$, we may estimate the right hand-side of (\ref{equation1.3.2}) as follows:
    \begin{align*}
        &\int_0^{\infty} \int_{\mathbb{R}^n} \left|u(t,x) (\partial_t^2 -\partial_t-\Delta)\big(\eta_R(t) \varphi_R(x)\big)\right| dxdt\\
        &\hspace{0.5cm} \leq \left(\int_0^\infty \int_{\mathbb{R}^n} |x|^{-\alpha}|u(t,x)|^p \Phi_R(t,x)dxdt\right)^{\frac{1}{p}}\\
        &\hspace{1cm}\times\left(\int_0^{\infty}\int_{\mathbb{R}^n}|x|^{\frac{\alpha p'}{p}}\left|\eta_R(t)\varphi_{R}(x)\right|^{-\frac{p'}{p}} \left|(\partial_t^2 -\partial_t-\Delta)\big(\eta_R(t) \varphi_R(x)\big)\right|^{p'} dxdt\right)^{\frac{1}{p'}} \\
        &\hspace{0.5cm}\leq C \mathcal{I}_R^{\frac{1}{p}} R^{\frac{n+2}{p'}+\frac{\alpha}{p}-2},
    \end{align*}
    where one can see that $|\Delta \varphi(x)| \lesssim \langle x\rangle^{-n-3-\alpha/(p-1)}$ and $|x|^{\alpha p'/p} \varphi(x) \lesssim \langle x \rangle^{-n-1}$ for all $\alpha \geq 0$.
Thus, it follows immediately from (\ref{equation1.3.2}) that
    \begin{align*}
\varepsilon\int_{\mathbb{R}^n} (u_0(x)+u_1(x))\varphi_R(x)dx &\leq C \mathcal{I}_R^{\frac{1}{p}} R^{\frac{n+2}{p'}+\frac{\alpha}{p}-2}- \mathcal{I}_R \notag\\
&\leq C_1 R^{n+2+\frac{\alpha}{p-1}-2p'}. 
    \end{align*}
Now, we consider the set 
            \begin{align*}
                \mathcal{A}(\beta) &:= \Big\{(u_0, u_1) \text{ such that }  \,\, u_j: \mathbb{R}^n \to \mathbb{R} \text{ with } j \in \{0,1\} \text{ fulfilling } \\
                &\hspace{1cm} u_0(x) + u_1(x) \gtrsim  \langle x\rangle^{-n(\frac{1}{2}+\frac{\beta}{n})}(\log(e+|x|))^{-1}\Big\}.  
            \end{align*}
The paper \cite[page 17]{ChenReissig2023} has shown that
\begin{align*}
    \mathcal{A}(\beta) \cap \big(\dot{H}^{-\beta} \times \dot{H}^{-\beta}\big) \ne \emptyset
\end{align*}
for all $\beta \in [0, n/2)$.
 Therefore, from (\ref{condition2.1}) we obtain
    \begin{align*}
        \int_{\mathbb{R}^n} (u_0(x) +u_1(x))\varphi_R(x) dx &\geq \int_{|x| \leq R/2} \big(u_0(x) +u_1(x)\big) dx\\
        &\geq C\int_{R/4 \leq |x| \leq R/2} |x|^{-n(\frac{1}{2}+\frac{\beta}{n})} (\log(e+|x|))^{-1} dx \\
        &\geq C R^{\frac{n}{2}-\beta} (\log R)^{-1}.
    \end{align*}
    This implies 
    \begin{align*}
       C \varepsilon R^{\frac{n}{2}-\beta} (\log R)^{-1} \leq C_1 R^{n+2+\frac{\alpha}{p-1}-2p'},
    \end{align*}
    that is,
    \begin{align}\label{Eq.1}
        \varepsilon  \lesssim R^{\frac{n}{2}+\frac{\alpha}{p-1}+\beta-\frac{2}{p-1}} \log R
    \end{align}
    for all $R \geq 1$.
    Moreover, the condition (\ref{condition2.2}) implies
    \begin{align*}
         \frac{n}{2}+\frac{\alpha}{p-1}+\beta - \frac{2}{p-1} < 0.
    \end{align*}
     Hence, letting $R \to \infty$ in the right-hand side of (\ref{Eq.1}) we immediately obtain a contradiction. This completes the proof of Theorem \ref{Theorem2}.

\subsection{Estimates for lifespan}
At first, let us recall the definition of the lifespan of solutions to \eqref{Main.Eq.1}.
\begin{definition}\label{Defn_Lifespan}
The quantity $T_{\alpha,\beta}(\varepsilon)$, which is defined by
\begin{align*}
T_{\alpha, \beta}(\varepsilon):= \sup \big\{T\in (0,\ity) : &\mbox{ There exists a unique local weak solution } u\notag \\
	&\,\,\, \mbox{to \eqref{Main.Eq.1} on } [0,T) \mbox{ in the sense of Definition \ref{defweaksolution_1}} \notag\\
    &\mbox{ with a fixed parameter }\varepsilon>0\big\},
\end{align*}
is called the lifespan (or the so-called maximum existence time) of solutions to the problem \eqref{Main.Eq.1}.
\end{definition}
In this subsection, we will summarize how to get some estimates for lifespan $T_{\alpha, \beta}(\varepsilon)$ of solutions to \eqref{Main.Eq.1} in the subcritical case $1< p< p_{\rm c}(\alpha,\beta,n)$, where $n = 1,2$ and $0 \leq \alpha, \beta < n/2$. More precisely, both lower bound and upper
bound estimates for $T_{\alpha,\beta}(\varepsilon)$ are given by the next propositions.

\begin{proposition}[Lower bound of lifespan]\label{Pro3.1}
Let $n=1,2$ and $0 \leq  \alpha, \beta < n/2$. The exponent $p$ fulfills
\begin{align}\label{condition3.1}
    1 + \frac{2(\beta-\alpha)}{n} < p < p_{\rm c}(\alpha,\beta,n).
\end{align}
Moreover, we assume that $(u_0, u_1) \in \mathcal{D}(\beta)$. Then, there exists a positive constant $\varepsilon_1$ such that for any $\varepsilon \in (0, \varepsilon_1]$,  the lower bound for the lifespan $T_{\alpha, \beta}(\varepsilon)$ can be estimated as follows: 
\begin{align}
    T_{\alpha, \beta}(\varepsilon) \gtrsim \varepsilon^{-\frac{4(p-1)}{4-2\alpha-(n+2\beta)(p-1)}}. \label{Es.Pro3.1}
\end{align}
\end{proposition}

\begin{proof}
   To prove Proposition \ref{Pro3.1}, we reuse the notations from Section \ref{section3}. In the same procedure as the proof of Theorem \ref{Theorem1}, by the aid of the condition (\ref{condition3.1}), we derive the following estimates for all $u, v \in X(T, M\varepsilon)$:
    \begin{align*}
        \|\Phi[u]\|_{X(T)} &\leq C_1\varepsilon \|(u_0, u_1)\|_{\mathcal{D}(\beta)}+C_1(1+T)^{1-\frac{n}{4}(p-1)-\frac{(p-1)\beta}{2}-\frac{\alpha}{2}} \|u\|_{X(T)}^p\\
        & \leq \frac{M\varepsilon}{2} + C_1 M^p \varepsilon^p (1+T)^{1-\frac{n}{4}(p-1)-\frac{(p-1)\beta}{2}-\frac{\alpha}{2}}
    \end{align*}
    and 
    \begin{align*}
        \|\Phi[u]-\Phi[v]\|_{X(T)} &\leq C_2 (1+T)^{1-\frac{n}{4}(p-1)-\frac{(p-1)\beta}{2}-\frac{\alpha}{2}} \|u-v\|_{X(T)} \left(\|u\|_{X(T)}^{p-1} + \|v\|_{X(T)}^{p-1}\right)\\
        &\leq 2C_2 M^{p-1} \varepsilon^{p-1} (1+T)^{1-\frac{n}{4}(p-1)-\frac{(p-1)\beta}{2}-\frac{\alpha}{4}} \|u-v\|_{X(T)}.
    \end{align*}
    Now, if we assume that
\begin{align}\label{3.1.2}
    \max\{C_1, 2C_2\} M^{p-1} \varepsilon^{p-1} (1+T)^{1-\frac{n}{4}(p-1)-\frac{(p-1)\beta}{2}-\frac{\alpha}{2}} < \frac{1}{4},
\end{align}
then we may construct a unique local solution $u^* = \Phi[u^*] \in X(T , M\varepsilon)$ by an analogous argument to the proof of Theorem \ref{Theorem1}. Moreover, the following estimate holds:
    \begin{align*}
        \|u^*\|_{X(T)} < \frac{3M\varepsilon}{4}.
    \end{align*}
    Let us choose
    \begin{align*}
        T^* := \sup\left\{T \in (0, T_{\alpha, \beta}(\varepsilon)) \text{ such that } \mathcal{G}(T) := \|u^*\|_{X(T)} \leq M\varepsilon\right\}.
    \end{align*}
    If $T^*$ satisfies the inequality (\ref{3.1.2}), then from the previous estimate it follows that $\mathcal{G}(T^*) < (3M\varepsilon)/4$. Due to the fact that $\mathcal{G}(T)$ is a continuous and increasing function for any $T \in (0, T_{\alpha, \beta}(\varepsilon))$, there exists $T^0 \in (T^*, T_{\alpha, \beta}(\varepsilon))$ such that $\mathcal{G}(T^0) \leq M\varepsilon$. This contradicts to the definition of $T^*$. For this reason, one realizes
    \begin{align*}
        \max\{C_1, 2 C_2\} (1+T^*)^{1-\frac{n}{4}(p-1) -\frac{(p-1)\beta}{2}-\frac{\alpha}{2}} M^{p-1} \varepsilon^{p-1} \geq \frac{1}{4},
    \end{align*}
    that is,
    \begin{align*}
        T_{\alpha, \beta}(\varepsilon) \geq T^* \gtrsim \varepsilon^{-\frac{4(p-1)}{4-2\alpha -(n+2\beta)(p-1)}}.
    \end{align*}
     Thus, Proposition \ref{Pro3.1} has been proved.
\end{proof}

\begin{proposition}[\textbf{Upper bound of lifespan}]\label{upper-bound}
  Assume that the conditions of Theorem \ref{Theorem2} are satisfied. Then, there exists a positive constant $\varepsilon_2$ such that for any $\varepsilon \in (0, \varepsilon_2]$, the upper bound for the lifespan $T_{\alpha, \beta}(\varepsilon)$ can be estimated by
  \begin{align}\label{Es.Pro3.2}
      T_{\alpha, \beta}(\varepsilon) \lesssim \varepsilon^{-\frac{4(p-1)}{4-2\alpha-(n+2\beta)(p-1)}} (\log(\varepsilon^{-1}))^{\frac{4(p-1)}{4-2\alpha-(n+2\beta)(p-1)}}.
   \end{align}
   \end{proposition}
   \begin{proof}
     To verify the desired estimate in Proposition \ref{upper-bound}, we will repeat some steps in the proof of Theorem \ref{Theorem2} to obtain
       \begin{align*}
           \varepsilon \leq C R^{\frac{n}{2}+\beta+\frac{\alpha}{p-1} -\frac{2}{p-1} +\delta},
       \end{align*}
       for all $R \geq 1$ and $\delta$ is an arbitrarily small positive constant. Letting $R \to \sqrt{T_{\alpha, \beta}(\varepsilon)}$ in the last inequality leads to
       \begin{align*}
           T_{\alpha, \beta}(\varepsilon) \lesssim \varepsilon^{-\frac{2}{\delta(\alpha, \beta)}},
       \end{align*}
       where 
       \begin{align*}
           \delta(\alpha, \beta) := -\frac{n}{2}-\beta -\frac{\alpha}{p-1}+\frac{2}{p-1}-\delta > 0.
       \end{align*}
       Using again estimate (\ref{Eq.1}) with $R \to \sqrt{T_{\alpha, \beta}(\varepsilon)}$, one has
       \begin{align*}
           \varepsilon &\leq C (T_{\alpha, \beta}(\varepsilon))^{\frac{1}{2}(\frac{n}{2}+\beta+\frac{\alpha}{p-1}-\frac{2}{p-1})} \log(T_{\alpha, \beta}(\varepsilon)) \\
           &\leq \frac{2C_1}{\delta(\alpha, \beta)}(T_{\alpha, \beta}(\varepsilon))^{\frac{1}{2}(\frac{n}{2}+\beta+\frac{\alpha}{p-1}-\frac{2}{p-1})} \log(\varepsilon^{-1}),
       \end{align*}
       that is,
       \begin{align*}
           T_{\alpha, \beta}(\varepsilon) \lesssim \varepsilon^{-\frac{4(p-1)}{4-2\alpha-(n+2\beta)(p-1)}} (\log(\varepsilon^{-1}))^{\frac{4(p-1)}{4-2\alpha-(n+2\beta)(p-1)}}.
       \end{align*}
       Hence, the proof of Proposition \ref{upper-bound} is established.
   \end{proof}
   \begin{remark} \label{Remark3.1}
\fontshape{n}
\selectfont
    Linking the achieved estimates \eqref{Es.Pro3.1} and \eqref{Es.Pro3.2} in Propositions \ref{Pro3.1} and \ref{upper-bound} one recognizes that the lifespan estimates for solutions to \eqref{Main.Eq.1} with low dimension space in the subcritical case $1 < p< p_{\rm c}(\alpha, \beta,n)$ are determined by the following relations:
    $$\varepsilon^{-\frac{4(p-1)}{4-2\alpha-(n+2\beta)(p-1)}} \lesssim T_{\alpha, \beta}(\varepsilon) \lesssim \varepsilon^{-\frac{4(p-1)}{4-2\alpha-(n+2\beta)(p-1)}} (\log(\varepsilon^{-1}))^{\frac{4(p-1)}{4-2\alpha-(n+2\beta)(p-1)}}.$$
For this observation, we want to underline that this is an improvement in the upper bound estimate compared to  \cite[Theorem 1.8]{Ikeda2019} for $\alpha =\beta= 0$. Furthermore, we can see that the weight $|x|^{-\alpha}$ contributes to increasing the lifespan of blow-up solutions.
\end{remark}

%.........................................................
\section{Some corollaries}\label{Section4}
In this section, we will discuss the corollaries of Theorems \ref{Theorem1} and \ref{Theorem2}. Specifically, we focus on the initial data $(u_0, u_1) \in \dot{H}_m^{-\bar{\beta}} \times \dot{H}_m^{-\bar{\beta}}$ for $m \in (1,2]$ and $\bar{\beta} \geq 0$. By Proposition \ref{HLS_inequality},  it holds
\begin{align*}
    \||\nabla|^{-\bar{\beta}} \varphi\|_{\dot{H}^{-\theta(m)}}  \lesssim \| |\nabla|^{-\bar{\beta}} \varphi\|_{L^m} \text{ with } \theta(m) := n\left(\frac{1}{m}-\frac{1}{2}\right),
\end{align*}
that is, $$\dot{H}_m^{-\bar{\beta}} \hookrightarrow \dot{H}^{-\bar{\beta}-\theta(m)}.$$
In other words, $\dot{H}_m^{-\bar{\beta}} \subset \dot{H}^{-\bar{\beta}-\theta(m)}$ for all $\bar{\beta} \geq 0$. Therefore, we obtain the following corollary for the global existence (for simplicity, we only consider the cases $n =1,2$). 
\begin{corollary}\label{Corollary4.1}
Let $n=1,2$, $ m \in (1,2]$, $0 \leq \alpha < n/2$ and $0 \leq \bar{\beta} < n(m-1)/m$. We assume that the exponent $p$ satisfies  
        \begin{align*}
            p \geq 1+ \frac{m(2-\alpha)}{n+m\bar{\beta}}.
        \end{align*}
        In addition, the initial data satisfies
            $$(u_0, u_1) \in \mathcal{D}(\bar{\beta}, m) := (H^1 \cap \dot{H}_m^{-\bar{\beta}})\times (L^2 \cap \dot{H}_m^{-\bar{\beta}}).$$
       Then, there exists a constant $\varepsilon_0 > 0$ such that for any $\varepsilon \in (0, \varepsilon_0]$, the Cauchy problem \eqref{Main.Eq.1} admits a unique global (in time) solution 
        $
            u \in \mathcal{C}([0, \infty), H^1).
        $
      Furthermore, the following estimates hold for all $t > 0$:
        \begin{align*}
            \|u(t,\cdot)\|_{L^2} &\lesssim \varepsilon (1+t)^{-\frac{n}{2}(\frac{1}{m}-\frac{1}{2})+\frac{\bar{\beta}}{2}} \|(u_0,u_1)\|_{\mathcal{D}(\bar{\beta},m)},\\
            \|\nabla u(t,\cdot)\|_{L^2} &\lesssim \varepsilon (1+t)^{-\frac{n}{2}(\frac{1}{m}-\frac{1}{2})+\frac{1+\bar{\beta}}{2}} \|(u_0,u_1)\|_{\mathcal{D}(\bar{\beta}, m)}.
        \end{align*}    
\end{corollary}
\begin{proof}
    Following the approach presented above, it suffices to apply Theorem \ref{Theorem1} with $\beta = \bar{\beta} +\theta(m)$. Then, the condition (\ref{critical_exponent}) becomes
    \begin{align*}
        p \geq p_c(\alpha, \bar{\beta}+\theta(m), n) = 1+\frac{m(2-\alpha)}{n+m\bar{\beta}}
    \end{align*}
    and the data space
   $
        \mathcal{D}(\bar{\beta}, m) \subset \mathcal{D}(\bar{\beta} +\theta(m)).
    $
    Finally, $\beta \in [0, n/2)$ leads to $\bar{\beta} \in [0, n(m-1)/m)$.
\end{proof}

By a similar argument, we also arrive at the following corollary for the blow-up result.

\begin{corollary}\label{Corollary4.2}
    Let $n \geq 1, m \in (1,2], \,\, 0 \leq \alpha < 2$ and $0 \leq  \bar{\beta} < n(m-1)/m$. Consider that $(u_0, u_1) \in \dot{H}_m^{-\bar{\beta}} \times \dot{H}_m^{-\bar{\beta}} $ satisfies
        \begin{align*}
         u_0(x) + u_1(x) \gtrsim \langle x \rangle^{-n(\frac{1}{m}+\frac{\bar{\beta}}{n})} \log(e+|x|)^{-1}.  
        \end{align*}
        Moreover, the exponent $p$ satisfies the following condition:
        \begin{align*}
            1 < p < 1+\frac{m(2-\alpha)}{n+m\bar{\beta}}.
        \end{align*}
        Then, there is no global weak solution to \eqref{Main.Eq.1} in $\mathcal{C}([0, \infty), L^2)$.
\end{corollary}

From the statements of Corollaries \ref{Corollary4.1} and \ref{Corollary4.2}, we conclude that the value $$p_c(\alpha, \bar{\beta}+\theta(m), n) =1+ \frac{m(2-\alpha)}{n+m\bar{\beta}} $$ is the critical exponent for the problem (\ref{Main.Eq.1}) provided that the initial data $(u_0, u_1) \in \dot{H}_m^{-\bar{\beta}} \times \dot{H}_m^{-\bar{\beta}}$.
%==================================================================

% ------------------------------------------------------------------------

\subsection*{Acknowledgment}
This work is supported by Vietnam Ministry of Education and Training and Vietnam Institute for Advanced Study in Mathematics under grant number B2026-CTT-04.
\medskip

\noindent\textbf{Data availability statements}

\noindent My manuscript has no associated data. No new data were created during the study.
\medskip

\noindent\textbf{Conflict of interest}

\noindent The authors declare that they have no conflict of interest.

\end{document}